\documentclass{amsart}
\usepackage{verbatim}
\usepackage{ulem}
\usepackage{multirow}
\usepackage{amsfonts}
\usepackage{amsmath}
\usepackage{amssymb}
\usepackage[matrix,arrow]{xy}
\usepackage[colorlinks=true, allcolors=blue]{hyperref}
\usepackage[usenames]{color}

\usepackage{xcolor}

\newcommand{\matriz}[4]{\begin{pmatrix} #1 & #2\\ #3 & #4\end{pmatrix}}
\newcommand{\wh}[1]{\widehat{#1}}
\newcommand{\wt}[1]{\widetilde{#1}}

\newtheorem{theorem}{Theorem}
\newtheorem{proposition}[theorem]{Proposition}
\newtheorem{lemma}[theorem]{Lemma}

\newtheorem*{example}{Example}
\theoremstyle{remark}

\theoremstyle{definition}

\begin{document}

\title{WEAK $G$-IDENTITIES FOR THE PAIR $(M_2( \mathbb{C}),sl_2( \mathbb{C}))$}

\author[C\'odamo]{Ramon C\'odamo}
\thanks{R. C\'odamo was financed in part by the Coordena\c c\~ao de
Aperfei\c{c}oamento de Pessoal de N\'{\i}vel Superior - Brasil (CAPES) - Finance Code 001}
\address{Department of Mathematics, UNICAMP, 13083-859 Campinas, SP,  Brazil}
\email{ramoncodamo@gmail.com}

\author[Koshlukov]{Plamen Koshlukov}
\thanks{P. Koshlukov was partially supported by FAPESP grant No.~2018/23690-6 and by CNPq grant No.~302238/2019-0}
\address{Department of Mathematics, UNICAMP, 13083-859 Campinas, SP,  Brazil}
\email{plamen@ime.unicamp.br}


\begin{abstract}
In this paper we study algebras acted on by a finite group $G$ and the corresponding $G$-identities.
Let $M_2( \mathbb{C})$ be the $2\times 2$ matrix algebra over the field of complex numbers $ \mathbb{C}$ and let $sl_2( \mathbb{C})$ be the Lie algebra of traceless matrices in $M_2( \mathbb{C})$. Assume that $G$ is a finite group acting as a group of automorphisms on $M_2( \mathbb{C})$. These groups were described in the Nineteenth century, they consist of the finite subgroups of $PGL_2( \mathbb{C})$, which are, up to conjugacy, the cyclic groups $ \mathbb{Z}_n$, the dihedral groups $D_n$ (of order $2n$), the alternating groups $ A_4$ and $A_5$, and the symmetric group $S_4$. The $G$-identities for $M_2( \mathbb{C})$ were described by Berele. The finite groups acting on $sl_2( \mathbb{C})$ are the same as those acting on $M_2( \mathbb{C})$. The $G$-identities for the Lie algebra of the traceless $sl_2( \mathbb{C})$ were obtained by Mortari and by  the second author. We study the weak $G$-identities of the pair $(M_2( \mathbb{C}), sl_2( \mathbb{C}))$, when $G$ is a finite group. Since every automorphism of the pair is an automorphism for $M_2( \mathbb{C})$, it follows from this that $G$ is one of the groups above. In this paper we  obtain bases of the weak $G$-identities for the pair $(M_2( \mathbb{C}), sl_2( \mathbb{C}))$ when $G$ is a finite group acting as a group of automorphisms.

\end{abstract}
\maketitle

\section*{Introduction}

Let $F$ be a field and $F\langle X\rangle$ be the free associative unital $F$-algebra freely generated by a countable infinite set $X=\{x_1,x_2,\ldots\}$. A polynomial $f(x_1,\ldots, x_n)\in F\langle X\rangle$ is a polynomial identity for an $F$-algebra $A$ if $f(a_1,\ldots,a_n)=0$ for every $a_1$, \dots, $a_n\in A$. If there exists a non-trivial identity for $A$, then we say $A$ is a PI-algebra. The set Id(A) of all polynomial identities of $A$ is a $T$-ideal, that is, it is an ideal invariant under all endomorphisms of the free associative algebra. A subset $S\subseteq F\langle X\rangle$ is a basis for the identities of $A$ if $Id(A)$ is the least $T$-ideal that contains $S$. 

The description of the polynomial identities of a given algebra is one of the most important (and difficult) problems in PI theory and bases, that is generating sets, are only known in a few cases. One of the most important algebras, for obvious reasons, is $M_n(F)$, the algebra of the $n\times n$ matrices; no basis of its polynomial identities is known if $n>2$, except for $n=3$, and $4$, when $F$ is a finite field \cite{Genov}, \cite{GenovSiderov}. For $n=2$ and $F$ a field of characteristic 0,  Razmyslov \cite{Razmyslov} obtained a finite basis of identities. Over finite fields,  Malt'sev and Kuzmin \cite{MK} described a finite basis for $M_2(F)$. On the other hand, if $F$ is an infinite field of characteristic $p>2$, a basis of identities for $M_2(F)$ was obtained in \cite{Koshlukov}; the proof given there relied heavily on invariant theory, and on the so-called weak identities. The weak identities (also called identities of representations) were introduced by Razmyslov in \cite{Razmyslov}, and he used these in order to describe the identities satisfied by the associative algebra $M_2(F)$, and also by the Lie algebra $sl_2(F)$ when $F$ is a field of characteristic 0. Let $A^{(-)}$ be the Lie algebra obtained from an associative algebra $A$ by defining $[a,b]=ab-ba$ and let  $L\subseteq A^{(-)}$ be a Lie subalgebra. Then an associative polynomial $f(x_1,\ldots, x_n)\in F\langle X\rangle$ is a weak polynomial identity (weak identity for short) for the pair $(A, L)$ if it vanishes under every substitution of the $x_i$ by elements of $L$. Weak identities can be defined analogously for Jordan algebras and other structures that can arise from associative algebras. They turned out to be a powerful tool in studying polynomial identities for such algebras. Currently, weak polynomial identities are objects of great interest in PI theory, an overview of the topic can be found in \cite{D_Weak_identities}.

If $\phi\colon G\rightarrow Aut(A)$ is a group homomorphism, where $Aut(A)$ is the group of automorphisms of $A$, then $\phi$ induces an action of $G$ on $A$, and we say that $A$ is a $G$-algebra. Let $F\langle X;G\rangle$ be the free associative $F$-algebra freely generated by the set $X\times G$. We denote the pairs $(x,g)\in X\times G$ by $x^g$ (or $g(x)$) and $x^e$ by $x$, where $e$ is the identity of $G$. The group $G$ acts naturally on the free associative algebra $F\langle X;G\rangle$ in the following way: given $g\in G$ and a monomial $m=x_1^{g_1}\cdots x_n^{g_n}\in F\langle X;G\rangle$, we have $m^g=x_1^{gg_1}\cdots x_n^{gg_n}$. Now, for each $f=\sum \alpha_im_i\in F\langle X;G\rangle$, we set $f^g=\sum\alpha_im_i^g$. The algebra $F\langle X;G\rangle$ satisfies the following universal property: given any associative $G$-algebra $A$ and any set-theoretic map $\alpha\colon X\rightarrow A$, there exists a unique $G$-algebra homomorphism $\widetilde{\alpha}\colon F\langle X; G\rangle \rightarrow A$ which extends $\alpha$, that is, a homomorphism such that $\widetilde{\alpha}(f^g)=\widetilde{\alpha}(f)^g$, for each $g\in G$ and $f\in F\langle X;G\rangle$. We say that $F\langle X;G \rangle$ is the free associative $G$-algebra freely generated by $X$. Analogously, we can define $L\langle X;G\rangle$, the free Lie $G$-algebra freely generated by $X$. The elements of $F\langle X;G \rangle$ are called $G$-polynomials, while the elements of $L\langle X;G \rangle$ are called Lie $G$-polynomials.

A $G$-polynomial $f(x_1^{g_1},\ldots,x_n^{g_n})$ is called a $G$-polynomial  identity (or $G$-identity) for $A$ if $f(a_1^{g_1},\ldots,a_n^{g_n})=0$ for every $a_1$, \dots, $a_n\in A$, that is, if $f$ lies in the intersection of all kernels of $G$-homomorphisms $\alpha\colon F\langle X;G\rangle \rightarrow A$. Of course, we can replace fields with (associative commutative and unital) rings in the above definitions. In analogy with the case of ordinary identities, we can define ideals of $G$-identities (or $G$-ideals), basis of $G$-identities, and so on. Motivated by Amitsur's paper on rings with involution \cite{Amitsur}, Kharchenko \cite{Kharchenco} initiated the study of rings satisfying $G$-identities. On the other hand, the study of algebras (over a field) satisfying $G$-identities, began with a work of Giambruno and Regev \cite{GR}. 

Let $sl_n( \mathbb{C})$ be the Lie algebra of the traceless matrices of order $n$ over the complex numbers. It follows from the Skolem--Noether theorem that the group of automorphisms of $M_n( \mathbb{C})$ is the projective linear group $PGL_n( \mathbb{C})=GL_n( \mathbb{C})/Z(GL_n( \mathbb{C}))$. If $n=2$, it is well known that this is also the automorphism group of the Lie algebra $sl_2( \mathbb{C})$, but if $n> 2$, the automorphisms of $sl_n$ do not consist only of conjugations \cite[Chapter IX, Section 5]{Jacobson}. Berele \cite{Berele} obtained a basis of the $G$-identities for $M_2( \mathbb{C})$, when $G$ is a finite group acting faithfully on $M_2( \mathbb{C})$, that is, $G$ is a finite subgroup of $PGL_2( \mathbb{C})$. These groups have been known for a long time, formally since around 1870, although in ancient Greece they knew a description of the regular polyhedra (recall that the group of rotations of the tetrahedron is $A_4$, that of the cube and of the octahedron is $S_4$, and that of the icosahedron and of the dodecahedron is $A_5$). They consist, up to conjugacy, of the cyclic groups $ \mathbb{Z}_n$, the dihedral groups $D_n$ (of order $2n$), the alternating groups  $A_4$ and $A_5$, and the symmetric group $S_4$. The formal proof was first given by Klein \cite{Klein} in 1884. The description of the finite subgroups of $PGL_2(F)$, up to conjugacy, is also known for every separably closed field $F$ \cite{B}.  As mentioned above, the finite subgroups of automorphisms for the Lie algebra $sl_2(F)$ are the same as those for $M_2(F)$. Considering each of these, Mortari and the second author \cite{KM}, obtained a basis of  $G$-identities for the Lie algebra $sl_2( \mathbb{C})$.

Let $A$ be a $G$-algebra, and $L\subseteq A^{(-)}$ be a Lie subalgebra invariant under $G$-action. The pair $(A,L)$ is sometimes called an associative-Lie $G$-pair. We say that a $G$-polynomial $f(x_1^{g_1},\ldots, x_n^{g_n})\in F\langle X;G\rangle$ is a weak $G$-polynomial identity (or weak $G$-identity) for the par $(A,L)$ if it vanishes on $L$ under every substitution of the $x_i$ by elements of $L$. In this paper we describe the weak $G$-identities for the pair $(M_2( \mathbb{C}),sl_2( \mathbb{C}))$, when $G$ is a finite group. There is a well known duality between actions and gradings by finite abelian groups, see for example \cite[Sect. 3.3, pp. 63-69]{GZ}. The proof in \cite{GZ} is for associative algebras, but it remains valid for non-associative algebras, and pairs, as well. This duality was necessary for the description of the $G$-identities for $M_2( \mathbb{C})$ and for $sl_2( \mathbb{C})$. The same phenomenon occurs in the description of the weak $G$-identities of the pair $(M_2( \mathbb{C}), sl_2( \mathbb{C}))$. As in the case of polynomial identities, the set $Id^G(A,L)$ of all weak $G$-identities of the pair $(A, L)$ is an ideal. It is invariant by all endomorphisms of the $G$-pair $(F\langle X;G\rangle, L\langle X;G\rangle)$, that is, every endomorphism of pairs commuting with the $G$-action (see \cite{Razmyslov2}). We say that $Id^G(A,L)$ is a $G$-ideal, and as in the case of ordinary identities, we can define a basis of weak $G$-identities. 

Consider a pair $(A,L)$ and let $Aut(A,L)$ be its group of automorphisms. Observe that if $\phi\colon G\rightarrow Aut(A,L)$ is a $G$-action and $H$ is its kernel, then $\overline{G}=G/H$ acts faithfully on $A$, with $a^{gH}=a^g$. Furthermore, the map $\gamma\colon F\langle X;G\rangle\rightarrow F\langle X;\overline{G}\rangle$ given by $\gamma(x^g)=x^{gH}$ is a $G$-epimorphism and $Id^G(A)=\gamma^{-1} (Id^{\overline{G}}(A))$. Thus, without loss of generality we can assume that $G$ acts faithfully on the pair.

\section{Preliminaries}

Let $A= \mathbb{C}[a_i,b_i,c_i\mid i\geq 1]$ be the commutative polynomial algebra in the variables $X=\{a_i,b_i,c_i\mid i\geq 1\}$ and let $G\subseteq Aut(M_2( \mathbb{C}))$ be a finite group of automorphisms acting faithfully on $M_2( \mathbb{C})$. Then $G\subseteq PGL_2( \mathbb{C})$ acts on $M_2(A)$ by conjugations. Consider the associative and Lie $G$-algebras, respectively, $R^G$, $S^G\subseteq M_2(A)$, generated over $ \mathbb{C}$ by:
\[
s_i=\matriz{a_i}{b_i}{c_i}{-a_i},\quad i\geq 1.
\]
It can be proved that the pair $(R, S)$ is the relatively free associative-Lie $G$-pair in the variety of $G$-pairs generated by $id^G(M_2 (\mathbb{C}),sl_2( \mathbb{C}))$, and the proof is the same as that of the case of associative algebras, see for example \cite[Theorem 1.1.4]{GZ}. From now on, we will always assume that $G\subseteq Aut(M_2( \mathbb{C}),sl_2( \mathbb{C}))=Aut(M_2( \mathbb{C}))$ is a finite group acting faithfully on the pair $Aut(M_2(  \mathbb{C}),sl_2( \mathbb{C}))$, that is, $G$ is one of the groups $ \mathbb{Z}_n$, $D_n$, $A_4$, $A_5$ or $S_4$. The following lemma is proved in the same way as Lemma 2 of \cite{Berele}. Its proof is elementary and consists in observing that conjugation by an invertible matrix is an automorphism.

\begin{lemma}\label{mudar}
Let $B\in M_2(\mathbb{C})$ be an invertible matrix, and let $\phi\colon G\rightarrow PGL_2(\mathbb{C})$ and $\psi\colon G\rightarrow PGL_2(\mathbb{C})$ be actions of a group $G$ on the $G$-pair $(M_2(\mathbb{C}),sl_2(\mathbb{C}))$ such that $\phi(g)=B\psi(g)B^{-1}$ for every $g\in G$. Then $(M_2 ( \mathbb{C}),sl_2( \mathbb{C}))$ satisfies the same $G$-identities with respect to each of these actions.
\end{lemma}

Let $G=\{g_1, \ldots, g_n\}$ be a finite abelian group of order $n$ and denote by $\widehat{G}=\{\chi_1,\ldots,\chi_n\}$ the set of all its irreducible characters over $\mathbb{C}$. Since $G$ is abelian, $\widehat{G}$ is a group with product given by $\chi_1\cdot\chi_2(g)=\chi_1(g)\cdot\chi_2(g)$ and $G\cong \widehat{G}$, see for example \cite[pp. 439]{Kostrikin}.
Let $A$ be a $G$-algebra and let $e_1$, \dots, $e_n$ be the minimal idempotents of the group algebra $ \mathbb{C} G$. One can choose $e_i= (1/n)\sum_{k=1}^n \chi_i(g_k^{-1})g_k$. Consider the subspaces $A^{(\chi_i )}=\{a\in A\mid a^g=\chi_i(g)a,\,g\in G\}$. Since $1=\sum_{i=1}^n e_i$ we have $A=\sum_{\chi\in \wh{G}} A^{(\chi)}$. Indeed, it is a $\wh{G}$-grading for $A$ and $A^{(\chi_i)}=\{a\in A\mid a^{e_i}=a\}$. Now, suppose that $A=\oplus_{g\in G}A^{(g)}$ is a $G$-grading on $A$. Given an irreducible character $\chi\in \wh{G}$, we define a function (with certain abuse of notation we denote it by the same letter $\chi$), $\chi\colon A\rightarrow A\in Aut(A)$ as follows:
 \[
 \chi(a)=\sum \chi(g)a_g,
\]
where $a_g\in A^{(g)}$ for each $g\in G$. It can be readily seen that this defines a $G$-action on $A$. Hence a $G$-grading gives us a $\wh{G}$-action and vice versa. For more details, see \cite[Theorem 3.2.1]{GMZ}.

\begin{example}\label{exemploBerele}
 Let $g$ be the automorphism of $M_2( \mathbb{C})$ given by $g(x)=BxB^{-1}$, where $B=e_{11}-e_{22}$. Thus $G=\{1,g\}\cong  \mathbb{Z}_2$ acts on $M_2( \mathbb{C})$ and this action induces the $ \mathbb{Z}_2$-grading $M_2( \mathbb{C})=M_2^{(\chi_1)}\oplus M_2^{(\chi_2)}$, where $M_2^{(\chi_1)}=\{a\in M_2( \mathbb{C})\mid g(a)=a\}$ and $M_2^{(\chi_2)}=\{a\in M_2( \mathbb{C})\mid g(a)=-a\}$. Note that if $a=pe_{11}+qe_{12}+re_{21}+se_{22}$, then $g(a)=a$ occurs only when $q=r=0$. On the other hand, $g(a)=-a$ is equivalent to the equality $p=s=0$, that is, the $\wh{G}$-grading obtained above is the natural $ \mathbb{Z}_2$-grading on $M_2(\mathbb{C})$.

\end{example}

We consider again a finite abelian subgroup $G\subseteq Aut(A)$. By definition, we know that $G$ acts as an automorphisms group on the free $G$-algebra $ \mathbb{C}\langle X;G\rangle$ and therefore $ \mathbb{C}\langle X;G\rangle=\oplus_{\chi\in \wh{G}}  \mathbb{C}\langle X;G\rangle^{(\chi)}$ is a $\wh{G}$-grading. Considering the action and the grading mentioned above,  Giambruno, Mishchenko and Zaicev \cite{GMZ} obtained the following results.

\begin{theorem}\label{Dualidade1}
 Let $A$ be a $ \mathbb{C}$-algebra and let $G$ be a finite abelian group. Every $G$-grading on $A$ defines a $\wh{G}$-action on $A$ by automorphisms and vice versa. In this action, a vector subspace $V\subseteq A$ is a $G$-graded subspace if and only if $V$ is invariant under the $\wh{G}$-action. An element $a\in A$ is homogeneous in the  $G$-grading if and only if $a$ is an eigenvector for some $\chi \in \wh{G}$.
\end{theorem}

\begin{lemma}\label{acao_graduacao}
 Let $A$ be an algebra and let $\phi\colon  \mathbb{C}\langle X;G\rangle\rightarrow A$ be a homomorphism, where $G\subseteq Aut(A)$ is a finite abelian group. Then the following conditions are equivalent:
 \begin{itemize}
  \item [a)] $\phi$ is a homomorphism of $G$-algebras;
  \item [b)] $\phi$ is a homomorphism of $\wh{G}$-graded algebras.
 \end{itemize}
 \end{lemma}

As a consequence of Theorem \ref{Dualidade1} and Lemma \ref{acao_graduacao}, the authors of \cite{GMZ} proved the following result.
\begin{proposition}\label{Dualidade2}
 Let $G\subseteq Aut(A)$ be a finite abelian group and consider the $\wh{G}$-grading $ \mathbb{C}\langle X;G\rangle=\oplus_{\chi\in \wh{G}}  \mathbb{C}\langle X;G\rangle^{(\chi)}$. Then $ \mathbb{C}\langle X;G\rangle$ is the free $\wh{G}$-graded algebra of countable rank, freely generated by $\wt{X}=\{x^{e_i}\mid x\in X, i=1,\ldots,n\}$. Furthermore $id^G(A)=id^{gr}(A)$, where $id^{gr}(A)$ denotes the ideal of $\wh{G}$-graded identities of $A$.
\end{proposition}

If $A=\sum_{g\in G}A^{(g)}$ is a $G$-graded algebra and $L\subseteq A^{(-)}$ is a Lie subalgebra of $A^{(-)}$ that inherits the grading of $A$, that is, $L=\sum_{g\in G}(L\cap A^{(g)})$, we saw that the pair $(A,L)$ is a $G$-graded pair. Now, if $A$ is a $G$-algebra, then a Lie subalgebra $L\subseteq A^{(-)}$ is invariant by the $G$-action if and only if $L$ inherits the $\wh{G}$-grading, that is, a pair $(A,L)$ is a $G$-pair if only if it is a $\wh{G}$-graded pair. Since $Aut(A, L)\subseteq Aut(A)$ and $L\langle X;G\rangle$ inherits the $\wh{G}$-grading of the free $G$-algebra $ \mathbb{C}\langle X;G\rangle=\oplus_{\chi\in \wh{G}}  \mathbb{C}\langle X;G\rangle^{(\chi)}$, Theorem \ref{Dualidade1} and Lemma \ref{acao_graduacao} remain valid if we consider pairs $(A,L)$ instead of algebras and the proof of the following proposition repeats verbatim that of Proposition \ref{Dualidade2}.

\begin{proposition}\label{Dualidade2_fraca}
 Let $G\subseteq Aut(A, L)$ be a finite abelian group and consider the free $G$-pair $( \mathbb{C}\langle X;G\rangle, L\langle X;G\rangle)$. Then $( \mathbb{C}\langle X;G\rangle, L\langle X;G\rangle)$ is the free $\wh{G}$-graded pair of countable rank, freely generated by the set $\wt{X}=\{x^{e_i}\mid x\in X, i=1,\ldots,n\}$. Moreover, $id^G(A,L)=id^{gr}(A,L)$, where $id^{gr}(A,L)$ denotes the ideal of $\wh{G}$-graded identities of the pair $(A,L)$.
\end{proposition}

Let $D$ be an infinite integral domain. The next proposition (see \cite{KC}) provides a basis of $ \mathbb{Z}_2$-graded identities for the pair $(M_2(D),sl_2(D))$ with the natural $ \mathbb{Z}_2$-grading, that is, $M_2(D)=M^{(0)}\oplus M^{(1)}$, where $M^{(0)}$ is the set of the $2\times 2$ diagonal matrices, $M^{(1)}$ is the set of off-diagonal matrices, and $sl_2(D)=(M^{(0)}\cap sl_2)\oplus (M^{(1)}\cap sl_2)$. Let $Y=\{y_1, y_2,\ldots\}$, $Z=\{z_1,z_2,\ldots\}$ be disjoint infinite sets and let $D\langle X\rangle$ be the free associative $ \mathbb{Z}_2$-graded algebra freely generated by $X=Y\cup Z$, where the elements of $Y$ and $Z$ are regarded as even (of degree 0) and odd (of degree 1), respectively. Sometimes, when no ambiguity arises, we will omit the subscripts of the variables $y_i$ and $z_i$.

\begin{proposition}\label{Prop}
 Let $D$ be an infinite unital domain, and let $M_2(D)$ be the algebra of the  $2\times 2$ matrices equipped with its  natural $ \mathbb{Z}_2$-grading. Then the following graded identities form a basis of the $ \mathbb{Z}_2$-graded identities for the pair $(M_2(D),sl_2(D))$.
 \begin{enumerate}
  \item $[y_1,y_2]$,
  \item $z_1z_2z_3-z_3z_2z_1$,
  \item $yz+zy=0$.
 \end{enumerate}
\end{proposition}
We recall that in \cite{KC} it was proved that it was also proved that the $ \mathbb{Z}_2$-graded identities for the pair $(M_2(D),sl_2(D))$ satisfy the Specht property. Since we will not need this theorem here, we refer the interested reader to \cite{KC}.

\section{The cyclic groups}
Here we follow ideas from \cite{Berele} and \cite{KM}. Suppose $G= \mathbb{Z}_n$ is generated by $g$. Since $g$ has order $n$, it follows that $g$ acts via conjugation by a diagonalizable matrix $\phi(g)$. Thus, by Lemma \ref{mudar}, we can assume that $\phi(g)=\alpha e_{11}+\beta e_{22}$, where $\alpha$, $\beta\neq 0$. However, multiplying $\phi(g)$ by non-zero scalars does not change the action and therefore, multiplying this matrix by $1/\alpha$, we can assume that $g$ acts by conjugation through a matrix of the form $\phi( g)=e_{11}+\omega e_{22}$, where $\omega$ is a primitive $n$-th root of $1$. Therefore $g$ acts on $sl_2( \mathbb{C})$ as follows:
\[
g\matriz{a}{b}{c}{-a}=\matriz{a}{\omega^{-1}b}{\omega c}{-a}.
\]
In particular, if $n=2$, the action is the following:
\[
e\matriz{a}{b}{c}{-a}=\matriz{a}{b}{c}{-a}, \qquad g\matriz{a}{b}{c}{-a}
= \matriz{a}{-b}{-c}{-a}.
\]
Define $\pi_0(x)=e(x)+g(x)$ and $\pi_1(x)=e(x)-g(x)$. From Proposition \ref{Dualidade2_fraca} and Proposition \ref{Prop}, we obtain:

\begin{proposition}
 The following $ \mathbb{Z}_2$-polynomials form a basis of the weak $ \mathbb{Z}_2$-identities for the pair $(M_2(\mathbb{C}),sl_2(\mathbb{C}))$:
 \begin{enumerate}
  \item $[\pi_0(x_1),\pi_0(x_2)]$,
  \item $\pi_1(x_1)\pi_1(x_2)\pi_1(x_3)-\pi_1(x_3)\pi_1(x_2)\pi_1(x_1)$,
  \item $\pi_0(x_1)\pi_1(x_2)+\pi_1(x_2)\pi_0(x_1)$.
 \end{enumerate}
\end{proposition}
Now we suppose $G= \mathbb{Z}_n$, $n \geq 3$. Since $G$ is abelian, the group algebra $\mathbb{C} G$ admits a decomposition into a direct sum of minimal ideals $\mathbb{C} G=\oplus  \mathbb{C} e_i$, where $e_i=1/|G|\sum \omega^{-ij}g^j$ is a primitive idempotent. 
Note that
\[
e_0\matriz{a}{b}{c}{-a}=\matriz{a}{0}{0}{-a},\quad e_1\matriz{a}{b}{c}{-a}=\matriz{0}{0}{c}{0},
\]
\[
e_{n-1}\matriz{a}{b}{c}{-a}=\matriz{0}{b}{0}{0}.
\]
In general we have that
\[
e_i\begin{pmatrix}
 a & b\\
 c & -a
\end{pmatrix}= \begin{pmatrix}
 (\sum \omega^{-ij})a & (\sum \omega^{-(i+1)j})b\\
 (\sum \omega^{-(i-1)j})c & -(\sum \omega^{-ij})a
\end{pmatrix}.
\]
But $\omega^{-i}=1$ if and only if $i=0$. On the other hand, $w^{-(i+1)}=1$ only when $i=n-1$. Finally, $\omega^{-(i-1)}=1$ is equivalent to $i=1$.
Therefore, since $\omega$ is a primitive root of $1$, each $\omega^j$ is either $1$ or it is a root of the polynomial $f(x)=1+x+\cdots +x^{n-1}$. This shows that $e_{i}$ vanishes on $sl_2( \mathbb{C})$ if $i \neq 0$, $1$, $n-1$. From now on, we will write $e_{-1}$ instead of $e_{n-1}$. 

\begin{lemma}\label{Lema6} Let $G= \mathbb{Z}_n$, $n\geq 3$ and $\alpha=\pm 1$. The following relations are $G$-identities for the pair $(M_2( \mathbb{C}),sl_2( \mathbb{C}))$:
\begin{enumerate}
 \item $e_0(x)+e_{1}(x)+e_{-1}(x)=x$,
 \item $e_\alpha(x_1)e_\alpha(x_2)=0$,
 \item $[e_0(x_1),e_0(x_2)]=0$,
 \item $e_\alpha(x_1)e_{-\alpha}(x_2)e_\alpha(x_3)=e_\alpha(x_3)e_{-\alpha}(x_2)e_\alpha(x_1)$,
 \item $e_0(x_1)e_\alpha(x_2)+e_\alpha(x_2)e_0(x_1)=0$.
  \end{enumerate}
\end{lemma}
\begin{proof}
The proof is straightforward, it consists of directly evaluating the corresponding $G$-polynomials, which is why we omit it.
\end{proof}

Let $B$ be the set of the following monomials: 
\[
e_0(x)^{\underline{n}}e_{\alpha}(x_{i_1})e_{-\alpha}(x_{j_1})\cdots e_{\alpha}(x_{i_k})\widehat{e_{-\alpha}(x_{j_k})},
\]
where $\alpha=\pm 1,\,e_0(x)^{\underline{n}}=e_0(x_1)^{n_1}\cdots e_0(x_r)^{n_r}$, $i_1\leq i_2\leq \cdots\leq i_k, \,j_1\leq j_2\leq\cdots\leq j_k$, and the "hat" $\widehat{e_{-\alpha}(x_{j_k})}$ means that $e_{-\alpha}(x_{j_k})$ can be missing. Denote by $I$ the $G$-ideal of weak $G$-identities generated by the identities of Lemma \ref{Lema6}.
\begin{lemma}\label{Lema7}
 Let B be as above. Then the quotient $\mathbb{C}\langle X;G\rangle/I$ is spanned by the set  $\{ m+I  \mid m\in B \}$.
\end{lemma}
 \begin{proof}
 First note that the identity $x=e_0(x)+e_{1}(x)+e_{-1}(x)$, or more generally $g^{j}x=e_0(x)+ \omega { ^j}e_1(x)+\omega^{-j}e_{-1}(x)$, shows us that the monomial
\[
  m=m(x_1^{g_1},\ldots,x_n^{g_n})
\]
can be written as a linear combination of monomials in $e_0(x_i)$, $e_1(x_i)$ and $e_{-1}(x_i)$ modulo $I$. Hence we can suppose that $m$ is a monomial in the variables  $e_0(x_i)$, $e_1(x_i)$ and $e_{-1}(x_i)$. From identity 5 of Lemma \ref{Lema6}, we can rewrite each of these monomials leaving each $e_0(x_i)$ on the left side of the remaining elements. Furthermore, from identity 3 of the same lemma we can write $m= e_0(x)^{\underline{n}}e_{\alpha_1}(x_{r_1})\cdots e_{\alpha_s}(x_{r_s}) $ modulo $I$, with $\alpha_i=\pm1$. Now, we can assume that elements of the form $e_{\alpha_p}(x_i)$ and $e_{\alpha_p}(x_j)$ are not adjacent (due to identity 2). Finally, identity 4 allows us to rewrite $m$ as:
\[
  m=e_0(x)^{\underline{n}}e_{\alpha}(x_{i_1})e_{-\alpha}(x_{j_1})\cdots(x_{i_k})\widehat{e_{-\alpha}(x_{j_k})}
\]
modulo $I$, where $i_1\leq i_2\leq \cdots\leq i_k$ and $j_1\leq j_2\leq \cdots\leq j_k$.
 \end{proof}

\begin{lemma}\label{Lema8}
 No non-trivial linear combination of elements of $B$ is a $G$-identity for the pair $(M_2( \mathbb{C}),sl_2( \mathbb{C}))$.
\end{lemma}
 \begin{proof}
 If we evaluate an element of the set $B$ on the generic matrices 
\[
X_i=\matriz{a_i}{0}{0}{-a_i}=e_0\matriz{a_i}{b_i}{c_i}{-a_i}\in S^G
\]
and the said element contains a factor of the form $e_{1}(x_i)$ or $e_{-1}(x_i)$, then the $G$-monomial vanishes. Therefore, for each linear combination of such monomials, the only factors that remain after the evaluation on the generic matrices $X_i$ are those of the form $e_0(x)^{\underline{n}}=e_0(x_1)^{n_1}\cdots e_0(x_k)^{n_k}$ and the evaluation results in  $a_1^{n_1}\cdots a_k^{n_k}e_{11}+(-1)^{n_1+\cdots+ n_k}a_1^{n_1}\cdots a_k^{n_k}e_{22}$. Observing that the $a_i$ are algebraically independent, we can recover each $e_0(x)^{\underline{n}}$ from its evaluation. Therefore a trivial linear combination of monomials of $B$ occurs only if all coefficients of the factors $e_0(x)^{\underline{n}}$ are equal to zero. Thus, we can assume that in these linear combinations there is no element of the form $e_0(x)$. Now we consider the following monomials of the types $(1)$, $(2)$, $(3)$ and $(4)$:
  \begin{enumerate}
   \item $m_1= e_{1}(x_{i_1})e_{-1}(x_{j_1})\cdots e_{1}(x_{i_k})e_{-1}(x_{j_k})$,
   \item $m_2= e_{1}(x_{i_1})e_{-1}(x_{j_1})\cdots e_{1}(x_{i_k})$,
   \item $m_3= e_{-1}(x_{i_1})e_{1}(x_{j_1})\cdots e_{-1}(x_{i_k})e_{1}(x_{j_k})$,
   
   \item $m_4= e_{-1}(x_{i_1})e_{1}(x_{j_1})\cdots e_{-1}(x_{i_k})$.
   \end{enumerate}
Let, for each $i\geq 1$,
\[
  X_i=\matriz{a_i}{b_i}{c_i}{-a_i}.
\]
Evaluating each of the monomials above on $X_{i_1},X_{j_1},\ldots,X_{j_k}$, we  obtain, respectively, the matrices: 
  \begin{enumerate}
   \item $M_1=c_{i_1}\cdots c_{i_k}b_{j_1}\cdots b_{j_k}e_{22}$,
   \item $M_2=c_{i_1}\cdots c_{i_k}b_{j_1}\cdots b_{j_{k-1}}e_{21}$,
   \item $M_3=b_{i_1}\cdots b_{i_k}c_{j_1}\cdots c_{j_k}e_{11}$,
   \item $M_4=b_{i_1}\cdots b_{i_k}c_{j_1}\cdots c_{j_{k-1}}e_{12}$. 
  \end{enumerate}
Hence it is sufficient to consider linear combinations where the monomials are of the same type $(i)$, $i\in\{1,2,3,4\}$. But it is obvious that the monomials that appear in linear combinations of elements of the same type $m_r$ are pairwise distinct. Thus $m_r=m'_r$ if and only if $ m_r(X_{i_1},X_{j_1},\ldots,\widehat{X_{j_k}})=m'_r(X_{i_1},X_{j_1},\ldots,\widehat{X_{j_k}})$. In this way, if $f=\sum \alpha_{ij} m^j_i$, where $\alpha_{ij}\in  \mathbb{C}$ and $m_i^j\in B$ are of the type $(i)$, annihilates on the pair $(M_2( \mathbb{C}), sl_2( \mathbb{C}))$, then $\sum \alpha_{i1}M^j_1=\sum \alpha_{i2}M^j_2=\sum \alpha_{i3}M^j_3=\sum \alpha_{i4}M^j_4=0$, where $M^j_i$ are the evaluations of $m^j_i$ on the generic matrices $X_i$. Since the variables are algebraically independent, it follows that $\alpha_{ij}=0$ for every $i$, $j$.
 \end{proof}

\begin{theorem}
 Let $G= \mathbb{Z}_n$, $n\geq 3$. The identities of Lemma \ref{Lema6} form a basis of the $G$-identities for the pair $(M_2( \mathbb{C}),sl_2( \mathbb{C}))$.
\end{theorem}
\begin{proof}
It is an immediate consequence of Lemmas \ref{Lema7} and \ref{Lema8}.
\end{proof}

\section{The dihedral groups}

Throughout this section we suppose that $G=D_n=\langle g,h\mid g^n=h^2=1, hgh=g^{-1}\rangle$, $n\geq 3$. In analogy with what was done in the case of the cyclic groups, we can assume that $g$ acts as before. Since $hgh=g^{-1}$, $h$ acts by conjugation via the matrix $h=e_{12}+e_{21}$, see \cite[Lemma 9]{Berele} or \cite[section 3]{KM}:
\[
g\matriz{a}{b}{c}{-a}=\matriz{a}{\omega^{-1}b}{\omega c}{-a},\qquad h\matriz{a}{b}{c}{-a}=\matriz{-a}{c}{b}{a}.
\]
It is well known (see, for example, \cite{Serre1977}), that in this case the group has only irreducible representations of degree 1 and 2. If $n=2m$ and $e_i$ are as before, then the group algebra $  \mathbb{C}G$ decomposes into a direct sum of copies of $ \mathbb{C}$ and of $M_2( \mathbb{C})$ as follows:
\begin{itemize}
 \item four copies of $ \mathbb{C}$, each generated by:
    \begin{itemize}
     \item [(a)] $(1/2)(1+h)e_0$,
     \item [(b)]$(1/2)(1-h)e_0$,
     \item [(c)]$(1/2)(1-h)(1+e_0)$,
     \item [(d)]$(1/2)(1+h)(1-e_0)$,
    \end{itemize}
\item $m-1$ copies of $M_2( \mathbb{C})$ generated by: $e_{\pm i}$, $he_{\pm i}$, $i\geq 1$.
    
\end{itemize}
On the other hand, if $n=2m+1$, then $ \mathbb{C}G$ decomposes as:
    \begin{itemize}
     \item two copies of $ \mathbb{C}$ generated by:
        \begin{itemize}
         \item [(a)] $(1/2)(1+h)e_0$,
         \item [(b)] $(1/2)(1-h)e_0$,
        \end{itemize}
    \item $m$ copies of $M_2( \mathbb{C})$ generated by:  $e_{\pm i}$, $he_{\pm i}$, $i\geq 1$.

    \end{itemize}
A direct computation shows that $e_{i}$ (and therefore $he_{i}$) annihilate $sl_2( \mathbb{C})$ if $i \neq 0$, $1$, $n-1$. As in the case of the cyclic group, we will write $e_{-1}$ instead of $e_{n-1}$.

\begin{lemma}\label{Lema9}
 Let $G=D_n$, $n\geq 3$ and $\alpha = \pm 1$. The following relations are $G$-identities for the pair $(M_2( \mathbb{C}),sl_2( \mathbb{C}))$:
 \begin{enumerate}
  \item $(1+h)e_0(x)=0$,
  \item $e_0(x_1)f(x_2)+f(x_2)e_0(x_1)=0$, where $f\in \{e_\alpha,he_\alpha\}$,
  \item $e_0(x)+e_1(x)+e_{-1}(x)=x$,
  \item $[e_0(x_1),e_0(x_2)]=0$,
  \item $f(x_1)f(x_2)=0$, where $\,f\in \{e_{\alpha},he_{\alpha}\}$,
  \item $f(x_1)g(x_2)=0$, if $(f,g)\in \{(e_{\alpha}, he_{-\alpha}),(he_{\alpha},e_{-\alpha})\}$,
  \item $f(x_1)x_2g(x_3)=f(x_3)x_2g(x_1)$, where $f$, $g\in \{e_\alpha,he_\alpha\}$,
  \item $e_\alpha(x_1)x_2 he_{-\alpha}(x_3)=he_{-\alpha}(x_3)x_2e_\alpha(x_1)$,
  \item $e_\alpha(x_1)e_{-\alpha}(x_2)=he_{-\alpha}(x_2)he_{\alpha}(x_1)$,
  \item $e_\alpha(x_1)he_\alpha(x_2)=e_\alpha(x_2)he_\alpha(x_1)$. 
 \end{enumerate}
\end{lemma}
\begin{proof}
 The proof consists of direct (and simple but tedious) evaluation.
\end{proof}

We will see that the above ten identities form a basis of the $G$-identities for the pair $(M_2( \mathbb{C}),sl_2( \mathbb{C}))$, when $G$ is the dihedral group. Note that this is not a minimal basis. For example, identities (9) and (10) are equivalent; (7) and (8) are too, as well as (5) and (6). We decided to keep the identities 6, 8 and 10 so that the arguments of Lemma \ref{Lema10} are more transparent.

Let $I$ be the $G$-ideal of weak identities generated by the identities of Lemma \ref{Lema9} and let $B$ be the set formed by all monomials of the type $e_0(x)^{\underline{n }}\widehat{u}$, where either $u=1$ or $u$ is a monomial in $\{e_{\pm 1},he_{\pm 1}\}$ such that:
    \begin{itemize}
     \item [($\mathcal{B}_1$)] $he_1$ and $he_{-1}$ do not occur simultaneously,
     \item [($\mathcal{B}_2$)] Each $he_\alpha$ appears only after the $e_{-\alpha}$,
     \item [($\mathcal{B}_3$)] Each $e_\alpha$ (respectively $he_\alpha$) can be immediately followed only by $e_{-\alpha}$ and $he_\alpha$ (respectively by $e_\alpha$),
     \item [($\mathcal{B}_4$)] If $i\leq j$ and $f\in\{e_\alpha, he_\alpha\}$, then $f(x_i)$ comes before  $f(x_j)$,
     \item [($\mathcal{B}_5$)] If $i<j$, then $u$ does not contain factors of the type $f(x_j)g(x_i)$ where $f$, $g\in\{e_\alpha,he_\alpha\}$.
    \end{itemize}
Given the conditions $(\mathcal{B}_1)$--$(\mathcal{B}_5)$, the only possibilities for the monomials $u$ are $u=1$ and $u=u_1u_2$, where $u_1$ and/or $u_2$ can be equal to 1, or come from the list below:  
\begin{itemize}
 \item $u_1=e_{\alpha}(x_{i_1})e_{-\alpha}(x_{j_1}) \cdots e_{\alpha}(x_{i_k})\widehat{e_{-\alpha}(x_{j_k})}$,
 \item $u_2=he_{\beta}(x_{l_1})e_{\beta}(x_{l_2}) \cdots he_{\beta}(x_{l_{r-1}})\widehat{e_{\beta}(x_{l_r})}$.
\end{itemize}
Here $\beta=\alpha$ and $l_s=i_{k+s}$ if $u_1$ ends in $e_\alpha(x_{i_k})$, on the other hand $\beta=-\alpha$ and $l_s=j_{k+s}$ if $u_1$ ends in $e_{-\alpha}(x_{j_k})$. Furthermore, in every case we have $i_1\leq i_2\leq \cdots$,  $j_1\leq j_2\leq \cdots$. 

\begin{lemma}\label{Lema10}
 Let the set $B$ be as above. Then $\mathbb{C}\langle X;G\rangle/I$ is spanned by the set $\{m+I\mid m\in B\}$.
\end{lemma}
 \begin{proof}
In analogy with the argument from the case of the cyclic groups, we have $x^{g^j} = e_0(x)+\omega^je_1(x)+\omega^{-j}e_{-1}(x)$. In addition, $e_0(x)=-he_0(x)$ modulo $I$, which gives us: 
\[
x^h =  he_0(x)+he_1(x)+he_{-1}(x)  = -e_0(x)+he_1(x)+he_{-1}(x).
\]
In this way, each $G$-monomial can be written as a linear combination of monomials in the $e_0(x_j)$, $e_\alpha(x_j)$ and $he_\alpha(x_j)$, modulo $I$. 

Let $m$ be a monomial in $e_0(x_j)$, $e_\alpha(x_j)$ and $he_\alpha(x_j)$.
Let us suppose that $m=m'he_\alpha(x_1)Yhe_{-\alpha}(x_2)m''\notin I$, where neither $he_1$ nor $he_{-1}$ appear in $Y$. If $Y=1$, by using identity (9) from Lemma \ref{Lema9}, we can write $m$ in the form $m= m'e_{-\alpha}(x_2)e_\alpha(x_1) m''$. Otherwise, according to identities (5) and (6), $Y$ must have the form $Y=e_\alpha(y_1)e_{-\alpha}(y_2)\cdots e_\alpha(y_{k-1}) e_{-\alpha}(y_k)$.

By using the identities (8) and (9) successively, we obtain
    \begin{align*}
	m & = m'he_\alpha(x_1)(e_\alpha(y_1)e_{-\alpha}(y_2)\cdots e_\alpha(y_{k-1})e_{-\alpha}(y_k))he_{-\alpha}(x_2)m'' \\
      & \stackrel{(8)}{=} m'he_\alpha(x_1)he_{-\alpha}(x_2)Y'm'' \\
      & \stackrel{(9)}{=} m'e_{-\alpha}(x_2)e_\alpha(x_1)Y'm'',
	\end{align*}
modulo $I$, where $Y'$ is a monomial in the $e_{\pm 1}$'s. Repeating this procedure several times we obtain $(\mathcal{B}_1)$. Now we suppose $m$ is the following monomial: 
\[
m=m_1'he_\alpha(x_1)Ye_{-\alpha}(x_2) m_2''
\]
with $he_{\alpha}$ as close as possible to $e_{ - \alpha}(x_2)$ (and $m$ satisfying $\mathcal{B}_1$), that is, we have in $m$ some $e_{-\alpha}$ after some variable $he_\alpha$, but in $Y$ there are no variables $h_{\alpha}$. Since $m\notin I$, from the identities (6) and (8) we obtain 
\[
Ye_{-\alpha}(x_2)=e_\alpha(y_1)e_{-\alpha}(y_2)\cdots e_{-\alpha}(y_{k-1})e_{\alpha}(y_k)e_{-\alpha}(x_2).
\] 
Therefore $m=m_1'Y'e_{-\alpha}(x_2)e_{\alpha}(y_{k})he_\alpha(x_1)m_2''$ 
modulo $I$. Here $Y'$ is a monomial in the $e_{\pm 1}$'s. This allows us to assume, without interfering with $\mathcal{B}_1$, that in $m$ each $he_\alpha$ occurs only after all $e_{-\alpha}$. Thus we obtain $(\mathcal{B}_1)$--$(\mathcal{B}_2)$. By identities (5) and (6) we already have:
\begin{align*}
	e_\alpha(x_1)e_\alpha(x_2) & = 0, \\
    e_\alpha(x_1)he_{-\alpha}(x_2) & =0, \\
    he_\alpha(x_1)he_\alpha(x_2) & = 0,\\
    he_\alpha(x_1)e_{-\alpha}(x_2)& =0,
\end{align*}
and, from what we saw above, we eliminate instances of the form $he_\alpha(x_1)he_{-\alpha}(x_2)$. Therefore, each monomial $m\notin I$ can be written modulo $I$ as a product of $e_0(x_i)$, $e_\alpha(x_i)$ and $he_\alpha(x_i)$, where each $e_\alpha$ is immediately followed only by $e_{-\alpha}$ or $he_{\alpha}$, while $he_\alpha$ is immediately followed by only $e_\alpha$. Furthermore, $he_{1}$ and $he_{-1}$ do not occur simultaneously. Therefore, the conditions $(\mathcal{B}_1)$--$(\mathcal{B}_3)$ are valid.
Let $f$, $g\in\{e_\alpha,he_\alpha\}$ and suppose that variables of the same type ($e_\alpha$ or $he_\alpha$) appear in $m$. Thus, $m$ has the form $m=m_2'f(x_i)m_2''f(x_j)m_3$. Choose $f(x_i)$ and $f(x_j)$ in such a way that in $m_2''$ there is no element of type $f$. If $f=e_\alpha$, then from the properties ($\mathcal{B}_1$)--($\mathcal{B}_3$) and the identity (6) it follows that $m_2''\in\{ e_{-\alpha}(y_1), he_{\alpha}(y_1)\}$. On the other hand, if $f=he_{\alpha}$, again from ($\mathcal{B}_1$)--($\mathcal{B}_3$) and from (6) we obtain $m_2''= e_{\alpha}(y_1)$. In any case, $m=m_2'f(x_j)m_2''f(x_i)m_3$ modulo $I$ (identity (7)), and this gives us $(\mathcal{B}_1)$--$( \mathcal{B}_4)$. If $m=m_3'f(x_j)g(x_i)m_3''$, the identity (10) allows us to write $m=m_3'f(x_i)g(x_j)m_3''$ modulo $I$, therefore the conditions $(\mathcal{B}_1)$--$(\mathcal{B}_5)$ are satisfied. Finally, all entries $e_0$ can be placed to the left in each monomial due to (2) and the variables $x_i$ of the elements of the form $e_0(x_i)$ can be placed in ascending order (identity (4)). In this way the lemma is proved.
\end{proof}

\begin{lemma}\label{Lema11}
No non-trivial linear combination of elements of $B$ is a $G$-identity for the pair $(M_2 ( \mathbb{C}), sl_2 ( \mathbb{C}))$.
\end{lemma}
 \begin{proof}
  Consider the generic matrices $X_i=a_i(e_{11}-e_{22})+b_ie_{12}+c_ie_{21}$ and $f$ a linear combination of elements of $B$. We can write $\displaystyle f={\textstyle\sum_{\underline{n}}}e_0(x)^{\underline{n}}u_{\underline{n}}$, where the $e_0(x)^{\underline{n}}$ are pairwise distinct. Each $u_{\underline{n}}$ is of the form 
  \[
  u_{\underline{n}}=\sum_{i,j=1}^2\sum_k \alpha_{ij}^k(\underline{n})u(\underline{n})_{ij}^k,
  \]
  with $u(\underline{n})_{ij}^k\in B$, $\alpha(\underline{n})_{ij}^k\in  \mathbb{C}$, $u(\underline{n})_{ij}^k(X_1,\ldots, X_m)=g(\underline{n})_{ij}^ke_{ij}$ where $g(\underline{n})_{ij}^k\in  \mathbb{C}[b_i,c_i\mid i\geq 1]$. When we evaluate $e_0(x)^{\underline{n}}$ in the generic matrices $X_1$, \dots, $X_m$ we obtain a matrix of the form $m(\underline{n})(e_{11} \pm e_{ 22})$, with $m(\underline{n})=a_1^{n_1}\cdots a_k^{n_k}$. Therefore   
  \begin{equation*}
    \begin{split}
	\displaystyle f(X_1,\ldots, X_m) & =  \sum_{\underline{n}}m(\underline{n})(e_{11}\pm e_{22})(\sum_{i,j=1}^2\sum_k \alpha_{ij}^k(\underline{n})u(\underline{n})_{ij}^k(X_1,\ldots, X_m)) \\
        & = \displaystyle \sum_{\underline{n}}m(\underline{n})\matriz{\displaystyle\sum_k \alpha_{11}^k(\underline{n})g(\underline{n})^k_{11}}{\displaystyle\sum_k \alpha_{12}^k(\underline{n})g(\underline{n})^k_{12}}{\displaystyle\pm\sum_k \alpha_{21}^k(\underline{n})g(\underline{n})^k_{21}}{\displaystyle\pm\sum_k \alpha_{22}^k(\underline{n})g(\underline{n})^k_{22}}.
	\end{split}
\end{equation*}
Due to the order imposed on the $a_i$'s we can recover each $e_0(x)^{\underline{n}}$ from its evaluation in the generic matrices, that is,
\[
e_0(x)^{\underline{n}}=e_0(x)^{\underline{n'}} \mbox{\textrm{ if and only if }} m(\underline{n})=m(\underline{n'}).
\]
Clearly we must have, for each $i$ and $j$,
\[
\sum_k \alpha_{ij}^k(\underline{n})g(\underline{n})^k_{ij}=0,
 \]
whenever $f(X_1,\ldots, X_m)=0$. Thus, all we have to do is check whether the elements $u\in B$, such that when evaluated in the generic matrices, yield multiples of the same elementary matrix $e_{ij}$, are linearly independent. Write $u=u_1u_2$ and denote the evaluation of $u$ in the generic matrices $X_i$ by $\overline{u}$. Note that the only ways we can obtain a multiple of an elementary matrix $e_{ij}$ are as follows:

\begin{flushleft}
$(e_{11})$:
\end{flushleft}
\begin{itemize}
 \item [a)] $u_1=1$,\\
 $u_2=he_{1}(x_{i_{1}})e_{1}(x_{i_2}) \cdots he_{1}(x_{i_{r-1}}){e_{1}(x_{i_{r}})}$,\\ 
 $\overline{u} = c_{i_{1}}c_{i_{2}}\cdots c_{i_{r-1}}c_{i_{r}}e_{11}$,
 
 \item [b)] $u_1=e_{-1}(x_{i_1})e_{1}(x_{j_1}) \cdots e_{1}(x_{j_{k-1}})e_{-1}(x_{i_k})$,\\
 $u_2=he_{-1}(x_{i_{k+1}})e_{-1}(x_{i_{k+2}}) \cdots he_{-1}(x_{i_{k+r-1}})$,\\ 
 $\overline{u} = b_{i_1}b_{i_2}\cdots b_{i_{k+r-1}}c_{j_1}c_{j_2}\cdots c_{j_{k-1}}e_{11}$,
 
 \item [c)] $u_1=e_{-1}(x_{i_1})e_{1}(x_{j_1}) \cdots e_{-1}(x_{i_{k}})e_{1}(x_{j_k})$\\
 $u_2=1$\\ 
 $\overline{u} = b_{i_1}b_{i_2}\cdots b_{i_k}c_{j_1}c_{j_2}\cdots c_{j_{k}}e_{11}$,
 
 \item [d)] $u_1=e_{-1}(x_{i_1})e_{1}(x_{j_1}) \cdots e_{-1}(x_{i_{k}})e_{1}(x_{j_k})$,\\
 $u_2=he_{1}(x_{j_{k+1}})e_{1}(x_{j_{k+2}}) \cdots he_{1}(x_{j_{k+r-1}})e_{1}(x_{j_{k+r}})$,\\
 $\overline{u} = b_{i_1}b_{i_2}\cdots b_{i_k}c_{j_1}c_{j_2}\cdots c_{j_{k+r-1}}c_{j_{k+r}}e_{11}$,
 
\end{itemize}
\begin{flushleft}
$(e_{12})$:
\end{flushleft}
\begin{itemize}
 \item [a)] $u_1=1$,\\
 $u_2=he_{1}(x_{i_{1}})e_{1}(x_{i_{2}}) \cdots he_{1}(x_{i_{r}})$,\\
 $\overline{u} = c_{i_{1}}c_{i_{2}}\cdots c_{i_{r}}e_{12}$,
 
 \item [b)] $u_1=e_{-1}(x_{i_1})e_{1}(x_{j_1}) \cdots e_{1}(x_{j_{k-1}})e_{-1}(x_{i_k})$,\\
 $u_2=1$,\\
 $\overline{u} = b_{i_1}b_{i_2}\cdots b_{i_k}c_{j_1}c_{j_2}\cdots c_{j_{k-1}}e_{12}$,

 \item [c)] $u_1=e_{-1}(x_{i_1})e_{1}(x_{j_1}) \cdots e_{-1}(x_{i_{k}})e_{1}(x_{j_k})$,
 \\$u_2=he_{1}(x_{j_{k+1}})e_{1}(x_{j_{k+2}}) \cdots he_{1}(x_{j_{k+r}})$,\\
 $\overline{u} = b_{i_1}b_{i_2}\cdots b_{i_k}c_{j_1}c_{j_2}\cdots c_{j_{k+r}}e_{12}$,
 
 \item [d)] $u_1=e_{-1}(x_{i_1})e_{1}(x_{j_1}) \cdots e_{1}(x_{j_{k-1}})e_{-1}(x_{i_k})$,\\
 $u_2=he_{-1}(x_{i_{k+1}})e_{-1}(x_{i_{k+2}}) \cdots he_{-1}(x_{i_{k+r-1}})e_{-1}(x_{i_{k+r}})$,\\
 $\overline{u} = b_{i_1}\cdots b_{i_{k+r}}c_{j_1}\cdots c_{j_{k-1}}e_{12}$,
 
\end{itemize}
\begin{flushleft}
 $(e_{21})$:
\end{flushleft}
\begin{itemize}
 \item [a)] $u_1=e_{1}(x_{i_1})e_{-1}(x_{j_1}) \cdots e_{-1}(x_{j_{k-1}})e_{1}(x_{i_k})$,\\
 $u_2=1$,\\
 $\overline{u} = c_{i_1}c_{i_2}\cdots c_{i_k}b_{j_1}b_{j_2}\cdots b_{j_{k-1}}e_{21}$,
 
 \item [b)] $u_1=e_{1}(x_{i_1})e_{-1}(x_{j_1}) \cdots e_{-1}(x_{j_{k-1}})e_{1}(x_{i_k})$,\\
 $u_2=he_{1}(x_{i_{k+1}})e_{1}(x_{i_{k+2}}) \cdots he_{1}(x_{i_{k+r-1}})e_{1}(x_{i_{k+r}})$,\\
 $\overline{u} = c_{i_1}c_{i_2}\cdots c_{i_{k+r-1}}c_{i_{k+r}}b_{j_1}b_{j_2}\cdots b_{j_{k-1}}e_{21}$,
 
 \item [c)] $u_1=e_{1}(x_{i_1})e_{-1}(x_{j_1}) \cdots e_{1}(x_{i_{k}})e_{-1}(x_{j_k})$,\\
 $u_2=he_{-1}(x_{j_{k+1}})e_{-1}(x_{j_{k+2}}) \cdots he_{-1}(x_{j_{k+r}})$,\\
 $\overline{u} = c_{i_1}c_{i_2}\cdots c_{i_k}b_{j_1}b_{j_2}\cdots  b_{j_{k+r}}e_{21}$,
 
 \item [d)] $u_1=1$,\\
 $u_2=he_{-1}(x_{i_{1}})e_{-1}(x_{i_{2}}) \cdots he_{-1}(x_{i_{r}})$,\\
 $\overline{u} = b_{i_{1}}b_{i_{2}}\cdots b_{i_{r}}e_{21}$,
 
\end{itemize}
\begin{flushleft}
 $(e_{22})$:
\end{flushleft}
\begin{itemize}
 \item [a)] $u_1=e_{1}(x_{i_1})e_{-1}(x_{j_1}) \cdots e_{-1}(x_{j_{k-1}})e_{1}(x_{i_k})$,\\
 $u_2=he_{1}(x_{i_{k+1}})e_{1}(x_{i_{k+2}}) \cdots he_{1}(x_{i_{k+r-1}})$,\\
 $\overline{u} = c_{i_1}c_{i_2}\cdots c_{i_{k+r-1}}b_{j_1}b_{j_2}\cdots b_{j_{k-1}}e_{22}$,
 
 \item [b)] $u_1=e_{1}(x_{i_1})e_{-1}(x_{j_1}) \cdots e_{1}(x_{i_{k}})e_{-1}(x_{j_k})$,\\
 $u_2=1$,\\
 $\overline{u} = c_{i_1}c_{i_2}\cdots c_{i_k}b_{j_1}b_{j_2}\cdots b_{j_{k}}e_{22}$,
 
 \item [c)] $u_1=e_{1}(x_{i_1})e_{-1}(x_{j_1}) \cdots e_{1}(x_{i_{k}})e_{-1}(x_{j_k})$,\\
 $u_2=he_{-1}(x_{j_{k+1}})e_{-1}(x_{j_{k+2}}) \cdots he_{-1}(x_{j_{k+r-1}})e_{-1}(x_{j_{k+r}})$,\\
 $\overline{u} = c_{i_1}c_{i_2}\cdots c_{i_k}b_{j_1}b_{j_2}\cdots b_{j_{k+r}}e_{22}$,
 
 \item [d)] $u_1=1$,\\
 $u_2=he_{-1}(x_{i_{1}})e_{-1}(x_{i_{2}}) \cdots he_{-1}(x_{i_{r-1}})e_{-1}(x_{i_{r}})$,\\
$\overline{u} = b_{i_{1}}b_{i_{2}}\cdots b_{i_{r}}e_{22}$,

\end{itemize}
In each case for the $e_{ij}$, note that no relation among $G$-monomials in $(a)$, $(b)$, $(c)$ and $(d)$ is possible. For instance,  consider the case $(e_{11})$. We have four possibilities: in the first the total degree of the variables $b$ is zero, in the second the total degree of variables of type $b$ is larger than the degree of variables of type $c$, in the third the degrees for the variables of type types $b$ and $c$ coincide, and, finally, in the fourth case, the total degree of variables of type $c$ is larger than the total degree of variables $b$. Similar behaviour occurs for each elementary matrix $e_{ij}$, so we have only to check whether $G$-monomials of types $(a), (b), (c)$ and $(d)$ are linearly independent separately. Since the variables $b_i$, $c_i$, with $i\geq 1$, are in ascending order, the monomials can be recovered from their evaluation in the generic matrices. Furthermore, since the variables $b_i$, $c_i$ are algebraically independent, no relationship between the corresponding elements is possible. This proves the lemma.
 \end{proof}

\begin{theorem}
 Let $G=D_n$. The $G$-identities of Lemma \ref{Lema9} form a basis of the $G$-identities for the pair $(M_2( \mathbb{C}), sl_2 ( \mathbb{C}))$.
\end{theorem}
\begin{proof}
 It is an immediate consequence of Lemmas \ref{Lema10} and \ref{Lema11}.
 \end{proof}

\section{The groups \texorpdfstring{$A_4$, $A_5$ and $S_4$}{}}
Throughout this section, $G$ stands for one of the groups $G=A_4$, $A_5$ or $S_4$. The following result can be found in \cite[Lemma 12]{Berele} or in \cite[Lemma 14]{KM}.
\begin{lemma}
Suppose that $G$ acts faithfully on $sl_2( \mathbb{C})$. Then $sl_2( \mathbb{C})$ is an irreducible $ \mathbb{C} G$-module.
\end{lemma}
The three cases for the group $G$ are similar and are treated simultaneously. The same phenomenon occurs when computing the identities with $G$-action for $M_2( \mathbb{C})$ and $sl_2( \mathbb{C})$. As noted in \cite{Berele} and in \cite{KM}, we can decompose $ \mathbb{C} G$ into a direct sum of an ideal $J_0$ that annihilates $sl_2( \mathbb{C})$ and a copy of $M_3( \mathbb{C})$. Let $v_1=e_{11}-e_{22}$, $v_2=e_{12}$ and $v_3=e_{21}$. As an algebra of linear transformations on the three-dimensional space $sl_2( \mathbb{C})$, $M_3( \mathbb{C})$ has basis $\{\epsilon_{ij}\mid 1\leq i,j \leq 3\}$ (the canonical basis) which satisfies $\epsilon_{ij}(v_k)=\delta_{jk}v_i$. Explicitly, this basis acts on $x=a(e_{11}-e_{22})+be_{12}+ce_{21}\in sl_2( \mathbb{C})$ as follows:
\[
\epsilon_{11}(x)=\matriz{a}{0}{0}{-a},\quad \epsilon_{12}(x)=\matriz{b}{0}{0}{-b},\quad\epsilon_{13}(x)=\matriz{c}{0}{0}{-c},
\]
\[
\epsilon_{21}(x)=\matriz{0}{a}{0}{0},\quad \epsilon_{22}(x)=\matriz{0}{b}{0}{0},\quad\epsilon_{23}(x)=\matriz{0}{c}{0}{0},
\]
\[
\epsilon_{31}(x)=\matriz{0}{0}{a}{0},\quad \epsilon_{32}(x)=\matriz{0}{0}{b}{0},\quad\epsilon_{33}(x)=\matriz{0}{0}{c}{0}.
\]

\begin{lemma}\label{Lema13}
Let $G=A_4$, $A_5$ or $S_4$. The following relations are $G$-identities for the pair $(M_2( \mathbb{C}),sl_2( \mathbb{C}))$:
\begin{enumerate}
 \item  $\epsilon_{11}(x)+\epsilon_{22}(x)+\epsilon_{33}(x)=x$,
 \item $[\epsilon_{1\alpha}(x_1),\epsilon_{1\beta}(x_2)]=0$, $\alpha$, $\beta\in\{1,2,3\}$,
 \item $\epsilon_{1\alpha}(x_1)$ anticommutes with $\epsilon_{ij}(x_2)$, $i=2$, $3$ and $j$, $\alpha\in\{1,2,3\}$,
 \item $\epsilon_{i\alpha}(x_1)\epsilon_{j\beta}(x_2)\epsilon_{i\gamma}(x_3)=\epsilon_{1\alpha}(x_1)\epsilon_{1\beta}(x_2)\epsilon_{i\gamma}(x_3)$, where $i$, $j=2$, $3$ $(i\neq j)$ and $\alpha$, $\beta$, $\gamma=1$, $2$, $3$,
 \item $\epsilon_{i\alpha}(x_1)\epsilon_{i\beta}(x_2)=0$, for $i=2$, $3$ and $\alpha$, $\beta\in \{1,2,3\}$,
 \item $\epsilon_{3\alpha}(x_1)\epsilon_{2\beta}(x_2)+\epsilon_{2\beta}(x_2)\epsilon_{3\alpha}(x_1)=\epsilon_{1\alpha}(x_1)\epsilon_{1\beta}(x_2)$, where $\alpha$, $\beta=1$, $2$, $3$,

\item $\epsilon_{1\alpha}(x_1)\epsilon_{j\beta}(x_2)=\epsilon_{1\beta}(x_2)\epsilon_{j\alpha}(x_1)$, with $j\in \{2,3\}$ and $\alpha$, $\beta\in\{1,2,3\}$,
\item $\epsilon_{2\alpha}(x_1)\epsilon_{3\beta}(x_2)=\epsilon_{2\beta}(x_2)\epsilon_{3\alpha}(x_1)$, with $\alpha$, $\beta\in\{1,2,3\}$.
\end{enumerate}
\end{lemma}
\begin{proof}
Direct (and simple) evaluation of each $G$-polynomial.
\end{proof}

Let $B$ be the set of monomials $u=u_1u_2$ such that:
\begin{itemize}
 \item [($\mathcal{B}_1$)] $u_1=\epsilon_{11}(x_1)^{l_1}\cdots\epsilon_{11}(x_k)^{l_k}\epsilon_{12}(x_1)^{m_1}\cdots\epsilon_{12}(x_k)^{m_k}\epsilon_{13}(x_1)^{n_1}\cdots \epsilon_{13}(x_k)^{n_k}$,
 \item [($\mathcal{B}_2$)] $u_2\in \{1,\epsilon_{2i}(x_\alpha), \epsilon_{3j}(x_\alpha),\epsilon_{2i}(x_\alpha)\epsilon_{3j}(x_\beta)\}$,
 \item [($\mathcal{B}_3$)] In $u_2$ if occurs $\epsilon_{ij}(x_\alpha)$, $i\in \{2,3\}$, then for each $\epsilon_{1s}(x_p)$ in $u_1$ we have $s<j$ or $s=j$ and $ p\leq \alpha$,
 \item [($\mathcal{B}_4$)] In $u_2=\epsilon_{2i}(x_\alpha)\epsilon_{3j}(x_\beta)$, we have $i< j$ or $i=j$ and $\alpha\leq \beta.$ 
\end{itemize}
\begin{lemma}\label{Lema14}
 If $I$ is the weak $G$-ideal generated by the identities of Lemma \ref{Lema13}, then $\mathbb{C}\langle X;G\rangle/I$ is spanned by the set $\{u+I\mid u\in B\}$.
\end{lemma}
 \begin{proof}
Since $x=\sum \epsilon_{ii}(x)$ and $\{\epsilon_{ij}\mid 1\leq i,j\leq 3\}$ generates an ideal of $ \mathbb{C} G$, then each $G$-monomial can be written, modulo $I$, as a linear combination of monomials in the $\epsilon_{ij}(x_\alpha)$. Given such a monomial, we can put, up to a sign, all terms $\epsilon_{1j}(x_k)$ on the left side (identity (3)), then to the right of the monomial we will have a product of factors of the types $\epsilon_{2i}(x_\alpha)$ and $\epsilon_{3j}(x_\beta)$. On the other hand, the identity (5) eliminates the occurrence of monomials of the form $\epsilon_{ij}(x_\alpha)\epsilon_{ik}(x_\beta), \,i\in \{2,3\}$, so the terms $\epsilon_{2j}$ and $\epsilon_{3k}$ appear alternately. Now, using the identities (2), (3) and (4) successively, we can write modulo $I$ each monomial $m\in  \mathbb{C}\langle X;G\rangle\setminus I$ in the form $m=u_1u_2$, where
  \begin{itemize}
    \item $u_1=\epsilon_{11}(x_1)^{l_1}\cdots\epsilon_{11}(x_k)^{l_k}\epsilon_{12}(x_1)^{m_1}\cdots\epsilon_{12}(x_k)^{m_k}\epsilon_{13}(x_1)^{n_1}\cdots \epsilon_{13}(x_k)^{n_k}$,
    
    \item $u_2\in \{1,\epsilon_{2i}(x_\alpha), \epsilon_{3j}(x_\alpha), \epsilon_{2i}(x_\alpha)\epsilon_{3j}(x_\beta), \epsilon_{3j}(x_\beta)\epsilon_{2i}(x_\alpha)\}.$ 
  \end{itemize}
Using identity (6) we can send the elements $\epsilon_{3i}(x_\alpha)$ to the right of the $\epsilon_{2j}(x_\beta)$. If $u_2=\epsilon_{2j}(x_\alpha)$ or $u_2=\epsilon_{3j}(x_\alpha)$, and $\epsilon_{1s}(x_p)$ occurs in $u_1$ we put (if necessary) these two factors side by side (identity (2)), then we permute the second indices using the identity (7) and then once again we use (2) to reorder the factors in $u_1$. On the other hand, if $u=u_1 u_2$ with $u_ 2=\epsilon_{2i}(x_\alpha)\epsilon_{3j}(x_\beta)$, we proceed as before in order to arrange the first element of $u_2$. Choose $\epsilon_{1p}(x_q)$ in $u_1$ with $p$ as large as possible such that $j\leq p$. Now choose among these factors $\epsilon_{1p}(x_q)$, the one with the largest $q$. If there is no such $p$, then we are done. Otherwise, we will have three situations to consider:
\begin{itemize}
 \item [a)] If $j=p$ and $q=\beta$ nothing needs to be done.
 \item [b)] If $j=p$ and $\beta < q$, we swap $x_\beta$ for $x_q$ (identity (2), (3) and (7)). Thus, we will obtain:
 $u=\widetilde{u_1}\epsilon_{2i}(x_\alpha)\epsilon_{3p}(x_q)$ and this concludes the argument, since $p<i$ or $i=p$, and $\beta<q\leq\alpha$ (because $\epsilon_{2i}(x_\alpha)$ satisfies ($\mathcal{B}_3$)).
 \item [c)] Suppose $j<p$. In the same way as we did in (b), let us change $j$ to $p$ and $x_\beta$ to $x_q$, obtaining:

    $u=\widetilde{u_1}\epsilon_{2i}(x_\alpha)\epsilon_{3p}(x_q)$, so $j<p\leq i$ and $\epsilon_{2i}(x_\alpha)$ continues to satisfy ($\mathcal{B}_3$).
   \end{itemize}
 Finally, use identity (8) in order to obtain ($\mathcal{B}_4$).
 \end{proof}

Here we establish some notations: 
 \begin{itemize}
  \item $\underline{n}:=(l_1,\ldots,l_k,m_1,\ldots,m_k,n_1,\ldots,n_k)$, $l_i$, $m_i$, $n_i\geq 0$,
  \item $u_1(x)^{\underline{n}}:=\prod_{i=1}^k\epsilon_{11}(x_i)^{l_i}\prod_{i=1}^k\epsilon_{12}(x_i)^{m_i}\prod_{i=1}^k\epsilon_{13}(x_i)^{n_i}$,
  \item $U_1^{\underline{n}}:=a_1^{l_1}\cdots a_k^{l_k}b_1^{m_1}\cdots b_k^{m_k}c_1^{n_1}\cdots c_k^{n_k}$,
  \item $U_2:=u_2(X_1,\ldots, X_k)$ the evaluation of $u_2$ in the generic matrices $X_i=a_i(e_{11}-e_{22})+b_ie_{12}+c_ie_{21}$. 
 \end{itemize}
 
\begin{lemma}\label{Lema15}
 No non-trivial linear combination of elements of $B$ is a $G$-identity for the pair $(M_2( \mathbb{C}), sl_2( \mathbb{C}))$.
\end{lemma}
 \begin{proof}
 If $E_j^1=a_j$, $E_j^2=b_j $ and $E_j^3=c_j$, then $U_2$ evaluates on $I$, $E_\alpha^ie_{12}$, $E_\alpha^ie_{21}$, or $E_\alpha^iE_\beta^je_{11}$ whenever $u_2$ equals, respectively, 1, $\epsilon_{2i}(x_\alpha)$, $\epsilon_{3i}(x_\alpha)$, or $\epsilon_{2i}(x_\alpha)\epsilon_{3j}(x_\beta)$. Define the following order: $E_{p}^s<E_{\alpha}^j$ if $s<j$ or $s=j$ and $p\leq\alpha$. This is the order induced by the lexicographic order on the set of pairs $(j, \alpha)$. It establishes the following linear order on the set $\{a_i,b_i,c_i\mid i\geq 1\}$: 
\[
a_i < a_{i+1}<b_j<b_{j+1}<c_k<c_{k+1},\qquad i,j,k\geq 1.
\]
Suppose there exists an identity of the form:
\[
f=\sum_{\underline{n}}\alpha_iu_1(x)^{\underline{n}}+\sum_{j,p,q,\underline{n}}\beta_iu_1(x)^{\underline{n}}u_{pq}^j(x),
\]
where the $u_{pq}^j$ denote the elements $u_2\in B$ such that $U_2$ is a multiple of the elementary matrix $e_{pq}$. Since the variables in $u_1=u_1(x)^{\underline{n}}$ are ordered, $u_1$ is uniquely determined by its evaluation on the generic matrices. Furthermore, when we evaluate $\sum_{\underline{n}}\alpha_iu_1(x)^{\underline{n}}$ we obtain:
\[
\matriz{\sum_{\underline{n}}\alpha_iU_1^{\underline{n}}}{0}{0}{\pm\sum_{\underline{n}}\alpha_iU_1^{\underline{n}}}.
\]
The entry $e_{22}$ does not otherwise occur in the evaluation of $f$, so $\alpha_i=0$, for every $i\geq 1$, and all we need to do is show that there are no identities $ f= \sum_{i,\underline{n}}\alpha_iu_1(x)^{\underline{n}}u_{pq}^i(x)$ with $(p,q)\in \{(1 , 2) ,(2,1),(1,1)\}$. These three cases are all similar. Let $u=u_1u_2\in \mathcal{B}$. If $\epsilon_{1s}(x_p)$ occurs in $u_1$ and $\epsilon_{ij}(x_\alpha)$ occurs in $u_2$, then by ($\mathcal{B}_3$) we have $s <j$ or $s=j$ and $p\leq \alpha$, that is, $E_{p}^s<E_{\alpha}^j$. Let us write $u(X_1,\ldots,X_k)=\pm m_1m_2\widehat{m_3}e_{pq}$, where $m_2$ and $m_3$ are the two largest variables occurring in $u$. Observe that according to the previous argument, the variables that occur in $U_2$ will be larger than those that occur in $U_1^{\underline{n}}$, so $m_2\widehat{m_3}$ corresponds to $u_2$ and $ m_1$ corresponds to $u_1$. Analogously, by ($\mathcal{B}_4$), if $m_2<m_3$, then $m_2$ and $m_3$ correspond, respectively, to the first and second variable of $u_2$. This shows that we can recover the elements $u\in \mathcal{B}$ if we know their evaluations. As $a_i$, $b_i$ and $c_i$ are algebraically independent, no non-trivial linear combination among these monomials can vanish.
 \end{proof}

In order to clarify what happened in the proof above, we consider the following examples.
\begin{example} Suppose that after evaluating $u$ in generic matrices we obtain
\[
u(X_1,\ldots,X_k)=\matriz{0}{a_ 1^2b_{10}c_2c_7}{0}{0}.
\]
We have $u_2=\epsilon_{2j}(x_\alpha)$, since we obtain a non-zero entry at position (1,2) only, and this is possible only for $\epsilon_{2j}(x_\alpha)$ (see the beginning of the proof). Furthermore, $a_1\leq a_1\leq b_{10}\leq c_2\leq c_7$. As $a_1$, $b_{10}$ and $c_2$ are less than $c_7$, these cannot occur in the evaluation of $u_2$, by $(\mathcal{B}_3)$. Therefore $u_2=\epsilon_{23}(x_7)$ and the remaining variables come from the evaluation of $u_1$, but we have already seen that $u_1$ is determined by its evaluation and, thus $u_1=\epsilon_{ 11} ( x_1)^2\epsilon_{12}(x_{10})\epsilon_{13}(x_2)$.
\end{example}

 \begin{example} 
Now suppose that after evaluating $u$ in generic matrices we obtain
\[
u(X_1,\ldots,X_k)=\matriz{a_ 1^2b_{10}c_2c_7}{0}{0}{0},
\]
then $u_2=\epsilon_{2i}(x_\alpha)\epsilon_{3j}(x_\beta)$, since we have a multiple of $e_{11}$. Ordering the variables, we obtain $a_1\leq a_1\leq b_{10}\leq c_2\leq c_7$. As $a_1$ and $b_{10}$ are less than $c_2$, this cannot happen when evaluating $u_2$. Furthermore, since $u_2$ has two factors, the only possibilities are $u_2=\epsilon_{23}(x_2)\epsilon_{33}(x_7)$ and $u_2=\epsilon_{23}(x_7)\epsilon_{ 33 }(x_2)$. The second of these does not occur, because of $(\mathcal{B}_4)$. Therefore $u_2=\epsilon_{23}(x_2)\epsilon_{33}(x_7)$ and $u_1=\epsilon_{11}(x_1)^2\epsilon_{12}(x_{10})$.
\end{example}

Lemmas \ref{Lema14} and \ref{Lema15} imply the following theorem.
\begin{theorem}
Let $G=A_4$, $A_5$ or $S_4$. The identities of Lemma \ref{Lema13} form a basis of the $G$-identities for the pair $(M_2( \mathbb{C}),sl_2( \mathbb{C}))$.
\end{theorem}

\end{document}